\documentclass[a4paper,11pt]{article}
\usepackage{amsfonts}
\usepackage[T1]{fontenc}
\usepackage{microtype}
\usepackage{makeidx}
\usepackage{setspace}
\usepackage{authblk}
\usepackage{graphicx}
\usepackage{xcolor}
\usepackage[all]{xy}
\usepackage{amsmath}
\usepackage{amssymb}
\usepackage{booktabs}
\usepackage{braket}
\usepackage{array}
\usepackage{tabularx}
\usepackage{multirow}
\usepackage{indentfirst}
\usepackage{cite}
\usepackage{stmaryrd}
\usepackage{faktor}
\usepackage{titling}
\usepackage{tikz}
\usepackage{dsfont}
\usepackage{url}
\usepackage{hyperref}

\usepackage{tkz-graph}

\usetikzlibrary{matrix,arrows,decorations.pathmorphing}
    \renewcommand{\leq}{\leqslant}
    
\usepackage{amsthm}
\theoremstyle{plain}
\newtheorem{thm}{Theorem}[section]
\newtheorem*{thm*}{Theorem}
\newtheorem{dfn}[thm]{Definition}
\newtheorem{lem}[thm]{Lemma}
\newtheorem{prop}[thm]{Proposition}
\newtheorem{cor}[thm]{Corollary}

\theoremstyle{remark}
\newtheorem{oss}[thm]{Remark}

\theoremstyle{definition}
\newtheorem{ex}[thm]{Example}

\DeclareMathOperator{\imm}{\mathrm{Im}}

\DeclareMathOperator{\SSS}{\mathbb{S}}
\DeclareMathOperator{\BBB}{\mathbb{B}}

\DeclareMathOperator{\N}{\mathbb{N}}

\DeclareMathOperator{\var}{\vartriangleleft}

\title{\bf{Linearly shellable complexes}}
\author{}
\author{Paolo Sentinelli\thanks{ Dipartimento di Matematica, Politecnico di Milano, Milan, Italy. \\ \href{mailto:paolosentinelli@gmail.com}{paolosentinelli@gmail.com}}}
\date{}
\begin{document}

\maketitle

\vspace{-4em}

\begin{abstract}
We introduce the class of linearly shellable pure simplicial complexes. The characterizing property is the existence of a labeling of their vertices such that all linear extensions of the Bruhat order on the set of facets are shelling orders. Coxeter complexes of weak intervals and lower Bruhat intervals of parabolic right quotients, as type--selected Coxeter complexes of lower Bruhat intervals of parabolic left quotients, are proved to be linearly shellable. We also introduce the notion of linear strong shellability.
\end{abstract}

$\,$

\section{Introduction}
Shelling orders arising from linear extensions of posets appear more and more often in literature. Maybe the first constructions of this type are the shelling orders of Coxeter complexes realized by A. Bj\"{o}rner in \cite{bjorner}, using linear extensions of the weak order of Coxeter groups. Many years later Ardila, Castillo and Samper in \cite{ardila} have proved that the linear extensions of the poset of bases of a matroid, ordered with the internal order, are shelling orders for the independence complex of the matroid. 
An intimate relation between linear extensions and shelling orders has been also enlightened in a recent preprint of S. Lacina \cite{lacina}.
Connections between linear extensions and shellings are extensively investigated by  D. Bolognini and the author in \cite{BoloSentiBulletin}, where the focus was on Coxeter-type results. In fact, a finite pure simplicial complex is uniquely determined by its set of facets and this set can be always considered as a set of Grassmannian permutations (see \cite[Section 3]{BoloSentiBulletin}). Hence it is natural to order a pure simplicial complex with the Bruhat order; in this vein, we can see the dual graph of a simplicial complex (i.e. its flip graph) as a Bruhat graph in the parabolic setting (as observed in \cite[Section 5]{BoloSentiBulletin}). 

In this paper we introduce the class of \emph{linearly shellable} pure simplicial complexes. The main property of these complexes is that there exists a way to give names to the vertices (natural numbers) such that every linear extension of the Bruhat order is a shelling order (see Theorem \ref{teorema tutti shelling}). As stated in the theorem, this is equivalent to have the lexicographic order as shelling order.
For example, it is immediate that the independence complex of a matroid is linearly shellable. A class of linearly shellable complexes in Coxeter groups arises as Coxeter complexes of finite subsets of a Coxeter group, such as intervals and finite order ideals of the weak order, and finite order ideals of parabolic quotients. These are studied in Section \ref{sezione Coxeter}; they are restrictions or type--selections of Coxeter complexes. 

The linear shellability of these complexes follows from the order preserving property for the Bruhat order of the projections $P^J : W \rightarrow W^J$ and a splitting of the Bruhat order along parabolic quotients due to V. Deodhar (see Theorem \ref{th deodhar} and its corollaries). Linear extensions of the right weak order on the set $X\subseteq W$ we are considering provide shelling orders of their Coxeter complex $C(X)$. To obtain the linear shellability we need the Bruhat order on a set of Grassmannian permutations in bijection with $X$; to do this we consider
a refinement of the poset $(X,\leq)$, where $\leq$ is the Bruhat order (see the proof of Theorem \ref{teorema lineare} and Theorem \ref{teorema tutti shelling}). Such a refinement is denoted by $\preccurlyeq$ (see also Figure \ref{Hasse2} and Remark \ref{osservazione Bruhat}).  We also consider type--selected Coxeter complexes for finite order ideal in the Bruhat order of quotients $W^J$, generalizing to Coxeter groups the result of \cite[Theorem 4.5]{BoloSentiBulletin} (see Theorem \ref{teorema selected}), and two--sided Coxeter complexes, as introduced in \cite{petersen}.
We collect here some of the linearly shellable simplicial complexes studied in the section:
\begin{enumerate}
 \item the Coxeter complex of an order ideal of the right weak order (Theorem \ref{teorema lineare}). These include the \emph{Coxeter cones} of J. Stembridge's paper \cite{stembridge},
    \item the Coxeter complex of an interval $[u,v]_R$ of the right weak order (Corollary \ref{corollario intervalli}),
    
    \item the Coxeter complex of a lower Bruhat interval of ${^JW}$ (Corollary \ref{corollario quozienti}). The Coxeter complex of ${^JW}$ is a \emph{quotient Coxeter complex} in M. Wachs' paper \cite{wachs}.
    
    \item The type--selected Coxeter complex of a lower Bruhat interval of $W^J$, for type $S\setminus J$ (Theorem \ref{teorema selected}). These complexes appear as barycentric subdivisions of \emph{Schubert matroids} in \cite{BoloSentiBulletin}.

    \item The \emph{two--sided Coxeter complex}, as introduced in T. Petersen's paper \cite{petersen} (Theorem  \ref{teorema two-side}).
\end{enumerate}

The last section of the paper is devoted to prove properties for linear strong shellability analogous to the one of Section \ref{sezione linear}. \emph{Strong shellability} has been introduced by Guo, Shen and Wu in \cite{coreani}; the authors have proved, among other results, that order ideals of the Bruhat order on Grassmannian permutations (\emph{pure shifted complexes}) and independence complexes of matroids are strongly shellable. We obtain our results on linear strong shellability as consequences of Proposition \ref{lemma che non sa nessuno strong}; in particular we introduce \emph{promotion} and \emph{evacuation} for strong shelling orders.

\section{Notation and 
preliminaries}
In this section we fix notation and recall some definitions useful for the rest of the paper. We refer to \cite{Stanley} for posets and to \cite{BB} for Coxeter groups.

Let $\mathbb{N}$ be the set of non--negative integers. For $n\in \mathbb{N}\setminus \{0\}$, we use the notation $[n]:=\left\{1,2,\ldots,n\right\}$. For a finite set $X$, we denote by $|X|$ its cardinality and by $\mathcal{P}(X)$ its power set, which is an abelian group with the operation given by symmetric difference $A+B:=(A\setminus B) \cup (B\setminus A)$, for all $A,B\subseteq X$. Given a function $f:X\rightarrow Y$, we let $\imm(f):=\{f(x):x\in X\}$. If $Y$ is a poset, then $\imm(f)\subseteq Y$ is an induced subposet.
We denote by $X^n$ the $n$-th power under Cartesian product, by $x_i$ the projection of $x\in X^n$ on the $i$-th factor. Sometimes we write $x_1\ldots x_n \in X^n$ instead of $(x_1,\ldots,x_n)\in X^n$.
For $k\in \N$, $k \leqslant |X|$, we define the $k$-th \emph{configuration space} of $X$ by
$$\mathrm{Conf}_k(X):=\left\{x\in X^k: x_i=x_j \Rightarrow i=j, \, \, \forall\, i,j\in [k]\right\},$$
and, if $<$ is a total order on $X$, the  $k$-th \emph{unordered configuration space} of $X$ by
$$X^k_<:=\left\{x\in X^k: i<j \Rightarrow x_i < x_j, \, \forall \,i,j\in [k]\right\}.$$ 

If $(X,\preccurlyeq)$ is a poset, then $X^n$ is the poset defined by letting $x \leqslant y$ if and only $x_i \preccurlyeq y_i$, for all $i\in [n]$ and $x,y\in X^n$. The set $[n]$ is a poset under the natural order; so, for $k \in \N\setminus \{0\}$, the set $[n]^k$ is considered to be a poset and $[n]^k_<$ an induced subposet. We denote by $\var$ a covering relation in a poset $P$, i.e. $x \var y$ if and only if $x<y$ and the interval $[x,y]:=\{z\in P: x \leq z \leq y\}$ is $\{x,y\}$. A \emph{linear extension} of a finite poset $P$ of cardinality $n$ is a function
$L: P\rightarrow [n]$
such that $x< y$ implies $L(x)<L(y)$, for all $x,y\in P$. 

\vspace{1em}

Now we recall some general results on Coxeter groups which we use in Section \ref{sezione Coxeter}. 
Let $(W,S)$ be a Coxeter system, i.e.\ a presentation of the group $W$ given by a set $S$ of involutive generators and relations encoded by a \emph{Coxeter matrix}
$m:S\times S \rightarrow \{1,2,...,\infty\}$ (see \cite[Ch.~1]{BB}). When $|S|<\infty$ the system $(W,S)$ is said to be of \emph{finite rank}. A Coxeter matrix over $S$ is a symmetric matrix which
satisfies the following conditions for all $s,t\in S$:
\begin{enumerate}
  \item $m(s,t)=1$ if and only if $s=t$;
  \item $m(s,t)\in \{2,3,...,\infty\}$ if $s\neq t$.
\end{enumerate}  The presentation
$(W,S)$ of the group $W$ is then the following:
$$\left\{
  \begin{array}{ll}
    \mathrm{generators}: & \hbox{$S$;} \\
    \mathrm{relations}: & \hbox{$(st)^{m(s,t)}=e$,}
  \end{array}
\right.$$ for all $s,t\in S$, where $e$ denotes the identity in $W$. The Coxeter matrix $m$ attains the value $\infty$ at $(s,t)$ to indicate that there is no relation
between the generators $s$ and $t$.
The class of words expressing an element of $W$ contains words of minimal length; the \emph{length function} $\ell: W \rightarrow
\mathbb{N}$ assigns to an element $w\in W$ such minimal length. The identity $e$ is represented by the empty word and then $\ell(e)=0$. A \emph{reduced word} or \emph{reduced expression} for an element $w\in W$ is a word of minimal length representing $w$. If $J\subseteq S$ and $v\in W$, we let
\begin{gather*} W^J:=\Set{w\in W:\ell(w)<\ell(ws)~\forall~s\in J},
\\ {^JW}:=\Set{w\in W:\ell(w)<\ell(sw)~\forall~s\in J},
\\ D_L(v):=\Set{s\in S:\ell(sv)<\ell(v)},
\\ D_R(v):=\Set{s\in S:\ell(vs)<\ell(v)}.
\end{gather*}
With $W_J$ we denote the subgroup of $W$ generated by $J\subseteq S$; such a group is usually called a \emph{parabolic subgroup} of $W$. In particular, $W_S=W$ and $W_\varnothing =
\Set{e}$. 

Given a Coxeter system $(W,S)$, we let $\leqslant_R$  and $\leqslant$ be the \emph{right weak order}  and the \emph{Bruhat order}  on $W$, respectively. The covering relations of the right weak order are characterized as follows: 
$u \var_R v$ if and only if $\ell(u)<\ell(v)$ and $u^{-1}v \in S$. 
The covering relations of the Bruhat order are characterized as follows:
$u \var v$ if and only if $\ell(u)=\ell(v)-1$ and $u^{-1}v \in T$, where $T$ is the set of reflections of $(W,S)$.
The posets $(W,\leqslant_R)$ and $(W,\leqslant)$ are graded with rank function $\ell$ and $(W,\leqslant_R) \hookrightarrow (W,\leqslant)$. Their intervals have finite cardinality.
For $J\subseteq S$, each element $w\in W$ factorizes uniquely as
$w=w^Jw_J$, where $w^J\in W^J$, $w_J\in W_J$ and $\ell(w)=\ell(w_J)+\ell(w^J)$; see~\cite[Proposition~2.4.4]{BB}. We consider the idempotent function $P^J:W
\rightarrow W$ defined by
\begin{equation*} P^J(w)=w^J,
\end{equation*} for all $w\in W$. This function is order preserving for the Bruhat order
(see \cite[Proposition~2.5.1]{BB}). 
We also have that
\begin{equation} \label{inclusione proiezioni}
    I\subseteq J \, \, \Rightarrow \, \, P^J \circ P^I =P^J,
\end{equation} for all $I,J\subseteq S$.
For $s\in S$ we let $$P^{(s)}:=P^{S\setminus \{s\}}.$$ 
The following theorem, due to V. Deodhar, is \cite[Theorem 2.6.1]{BB}.
\begin{thm} \label{th deodhar}
    Let $E\subseteq \mathcal{P}(S)$ and $I:=\bigcap_{J\in E}J$. Let $u\in W^I$ and $v\in W$. Then
    $$u\leq v \, \, \Leftrightarrow \, \, P^J(u) \leq P^J(v) \, \,\forall \, J \in E.$$
\end{thm}

\vspace{1em}
Let $J\subseteq S$. Since $\bigcap_{s\in S\setminus J}(S\setminus \{s\})=J$, we obtain as immediate corollary the next statement.
\begin{cor} \label{deodhar}
    Let $u,v\in W^J$. Then 
    $$u\leq v \, \, \Leftrightarrow \, \, P^{(s)}(u) \leq P^{(s)}(v) \, \, \, \forall \, s \in S\setminus J.$$
\end{cor} 
The following two corollaries lay at the core of the shellability of Coxeter complexes. For the first one see \cite[Corollary 2.6.2]{BB}. 
\begin{cor} \label{cor deodhar0}
    Let $u,v\in W$. Then 
   $$u\leq v \, \, \Leftrightarrow \, \, P^{(s)}(u) \leq P^{(s)}(v) \, \,\forall \, s \in D_R(u).$$
\end{cor}
The next corollary provides, for $J=\varnothing$, an isomorphism between the Cayley graph of $(W,S)$ and the flip graph of its Coxeter complex. The general version we state here is suitable for type--selected Coxeter complexes (see Section \ref{sezione Coxeter}).
\begin{cor} \label{cor deodhar1}
    Let $(W,S)$ be a Coxeter system and $ J\subseteq S$ such that $|S\setminus J|<\infty$. Let $u,v\in W^J$. Then 
    $$ |\{(P^{(s)}(u),s): s\in S\setminus J\}\cap \{(P^{(s)}(v),s): s\in S\setminus J\}| 
        = |S\setminus J|-1 $$ $$\Longleftrightarrow  \, \,  \mbox{$u\neq v$ and $P^{J \cup \{r\}}(u)=P^{J \cup \{r\}}(v)$ for some $r\in S\setminus J$.}$$ 
\end{cor}
\begin{proof}
   Let $u \neq v$ and $P^{J \cup \{r\}}(u)=P^{J \cup \{r\}}(v)$ for some $r\not \in J$. If $s\not \in J \cup \{r\}$ we have that  
   \begin{eqnarray*}
       P^{(s)}(u) &=& (P^{(s)}\circ P^{J \cup \{r\}})(u) \\
       && (P^{(s)}\circ P^{J \cup \{r\}})(v)=P^{(s)}(v).
   \end{eqnarray*} Moreover, since $u\neq v$, by Corollary \ref{deodhar} we obtain $P^{(r)}(u) \neq P^{(r)}(v)$.
 Now assume that there exists $r\not \in J$ such that 
 $$\{(P^{(s)}(u),s): s\in S\setminus J\}\cap \{(P^{(s)}(v),s): s\in S\setminus J\} $$ $$ = \{(P^{(s)}(u),s): s\not \in J\cup \{r\}\}.$$ This implies $u\neq v$. Since $\bigcap_{s\not \in J\cup \{r\}} (S\setminus \{s\})=J \cup \{r\}$ and $P^{(s)}P^{J\cup \{r\}}=P^{(s)}$ for all $s\not \in J\cup \{r\}$, by Theorem \ref{th deodhar} we have that $u^{J\cup \{r\}}\leq v$ and $v^{J\cup \{r\}}\leq u$. By the order--preserving property of the projections, we obtain $u^{J\cup \{r\}}\leq v^{J\cup \{r\}}$ and $v^{J\cup \{r\}}\leq u^{J\cup \{r\}}$, i.e. $P^{J \cup \{r\}}(u)=P^{J \cup \{r\}}(v)$.
\end{proof}


\vspace{1em}

We consider the symmetric group $S_n$ of order $n!$ as a Coxeter group, with generators given by simple transpositions $S:=\{s_1,\ldots,s_{n-1}\}$, where, in one--line notation, $s_i:=12\ldots (i+1)i\ldots n$, for all $i\in [n-1]$.
For a permutation $w\in S_n$ we have a bijection $$D_R(w)\simeq \{i \in [n-1]: w(i)>w(i+1)\},$$ 
and, for $I\subseteq [n-1]$, $$S_n^{\{s_i: \, \, i\in I\}}\simeq \{w\in S_n : i\in I\, \Rightarrow\, w(i)<w(i+1)\}.$$ 
The action of $P^{(s)}$ on a permutation is explained in \cite[Section 2.4]{BB}. For example, in $S_5$ we have that $P^{(s_2)}(35412)=35124$ and $P^{(s_3)}(35412)=34512$.
For $k\in [n-1]$ the poset $[n]^k_<$ is isomorphic to $S_n^{S\setminus \{s_k\}}$ with the Bruhat order (see e.g. \cite[Proposition 2.4.8]{BB}). Moreover $[n]^n_< \simeq S_n^S$. For this reason we write about the \emph{Bruhat order} on $[n]^k_<$. This partial order is also known as \emph{Gale order}.
For example, in $[8]^4_<$ we have  $1357 \leqslant 1468$ and  $1467 \nleqslant 2378$. Notice that $(S_n^{S\setminus \{s_k\}},\leq)=(S_n^{S\setminus \{s_k\}},\leq_L)$, for all $k\in [n-1]$ (see \cite[Exercise 3.2]{BB}), where $\leq_L$ is the left weak order. 
We denote by $<_{lex}$ the \emph{lexicographic order} on $[n]^k_<$; it is a linear extension of the Bruhat order.
We also repeatedly use the identification $$[n]^k_< \simeq \{X\subseteq [n]: |X|=k\},$$ where $[n]^0_<:=\{\varnothing\}$. Then, identifying $U:=\bigcup\limits_{k=0}^n[n]^k_<$ with $\mathcal{P}(X)$,  it makes sense to write $x \cap y$, $x \cup y$ and the symmetric difference $x+y$, for all $x,y \in U$.

\vspace{1em}

A pure simplicial complex $X$ of dimension $d\in \N$ on $n\in \N\setminus \{0\}$ vertices is uniquely determined by the set of its facets, which are sets of cardinality $d+1$. Let $V(X)$ be the set of vertices of $X$. Given a bijective function $\sigma : V(X)\rightarrow [n]$, we can identify $X$ with a subset of $[n]^{d+1}_<$. For example, in Figure \ref{Hasse0} we have a pure simplicial complex of dimension $3$ on $8$ vertices, with $20$ facets, i.e. a subset of $[8]^4_<$ of cardinality $20$. By this identification, a total order on the set of facets of the pure simplicial complex $X$ uniquely determines an element of a suitable configuration space of $[n]^{d+1}_<$. On the other hand, an element of any configuration space of $[n]^{d+1}_<$ uniquely determines a simplicial complex and a total order on its facets. We pursue this point of view in the following definition of shelling order.
\begin{dfn} Let $h\in \N\setminus \{0\}$.
    An element $C\in \mathrm{Conf}_h([n]^k_<)$ is a \emph{shelling order} if $i,j\in [h]$, with $i<j$, implies that
there exists $r<j$ such that
$C_r=C_j+\{a,b\}$, with $a\in C_j\setminus C_i$ and $b\in C_r \setminus C_j$.
\end{dfn}
A pure simplicial complex is \emph{shellable} if there exists a shelling order on its facets.
The following is \cite[Proposition 5.3]{BoloSentiBulletin}.
\begin{prop} \label{lemma che non sa nessuno}
Let $C \in \mathrm{Conf}_h([n]^k_<)$ be a shelling order, with $h \geqslant 3$. If $|C_{h-1} \cap C_h|<k-1$ then $(C_1,\ldots,C_h,C_{h-1})$ is a shelling order.
\end{prop}

\section{Linear shellability}\label{sezione linear}

In \cite{BoloSentiBulletin} it was affirmed that no linear extension of the Bruhat order on the facets of the Hachimori’s complex is a shelling order. This is true when the names of the vertices are chosen as in \cite{Hachimori}. On the other hand, if the names are as in Figure \ref{Hachimori}, i.e. they are changed according to the permutation $6457312$, then every linear extension of the Bruhat order on the facets of this complex is a shelling order. This is not an accident, but a feature of the class of pure simplical complexes which we introduce in the following definition. Recall that $V(X)$ stands for the set of vertices of the simplicial complex $X$.

\begin{dfn}\label{def lin shell}
    Let $X$ be a pure simplicial complex of dimension $d\in \N$ with $n\in \N\setminus \{0\}$ vertices.  
    We say that $X$ is \emph{linearly shellable} if there exists a bijective function $\sigma : V(X)\rightarrow [n]$ which induces a function 
     $\sigma_X : X\rightarrow [n]^{d+1}_<$ such that the induced subposet $\imm(\sigma_X) \subseteq [n]^{d+1}_<$
    has a linear extension which is a shelling order.
\end{dfn}
By \cite[Theorem 4.8]{samper}, complexes with the quasi--exchange property are linearly shellable. In particular, matroids and order ideals of the Bruhat order are linearly shellable. Notice that in \cite[Theorem 3.4]{BoloSentiBulletin} all linear extensions of the Bruhat order of a complex with the quasi--exchange property are shelling orders. In the following theorem we prove that, when we label with natural numbers the vertices of a pure simplicial complex, either no linear extension of the Bruhat order is a shelling order, or every linear extension is a shelling order.

\begin{figure} \begin{center}\begin{tikzpicture}
\matrix (a) [matrix of math nodes, column sep=0.3cm, row sep=0.2cm]{
        & 567    &    \\
   367     &     &   457    \\
   247     &     &   346    \\
        &     &   246     \\
   237      & 345    &   146     \\
   &        &   145 \\
   126 &          &  135  \\
        &  123       &    \\};

\foreach \i/\j in {1-2/2-1, 1-2/2-3, 2-1/3-1, 2-1/3-3, 2-3/3-1, 2-3/3-3, 3-1/5-1, 3-1/4-3, 3-3/5-2, 3-3/4-3, 4-3/5-3, 5-2/6-3,
5-3/6-3, 5-1/7-1, 5-1/7-3, 6-3/7-3, 7-1/8-2, 7-3/8-2}
    \draw (a-\i) -- (a-\j);
\end{tikzpicture} \caption{Hasse diagram of the Hachimori's example of a shellable but not extendably shellable simplicial complex. The linear extensions are shelling orders.} \label{Hachimori} \end{center} \end{figure}
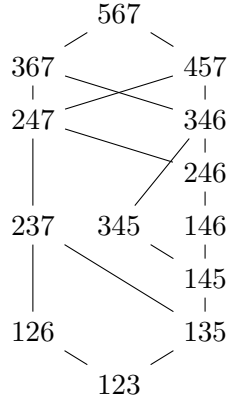

\begin{thm} \label{teorema tutti shelling}
 Let $X$ be a pure simplicial complex of dimension $d\in \N$ with $n\in \N\setminus \{0\}$ vertices. 
    The following are equivalent:
    \begin{enumerate}
        \item[1)] $X$ is linearly shellable;
        \item[2)] there exists a bijective function $\sigma : V(X)\rightarrow [n]$ which induces a function $\sigma_X: X\rightarrow [n]^{d+1}_<$ such that $<_{lex}$ on $\imm(\sigma_X)$ is a shelling order;
        \item[3)] there exists a bijective function $\sigma : V(X)\rightarrow [n]$ which induces a function $\sigma_X: X\rightarrow [n]^{d+1}_<$ such that every linear extension of $\imm(\sigma_X)$ is a shelling order. 
    \end{enumerate}
\end{thm}
\begin{proof}
    Clearly $3)$ implies $2)$, since $<_{lex}$ is a linear extension of the Bruhat order on $[n]^k_<$. It is also clear, by Definition \ref{def lin shell}, that $2)$ implies $1)$. It remains to prove that $1)$ implies $3)$. Let $X$ be linearly shellable and $l:=(y_1,\ldots,y_h)$ a linear extension of $Y:=\imm(\sigma_X)$ which is a shelling order. 
    If $h=1$ there is nothing to prove. If $h=2$, $Y$ is a chain of cardinality $2$, and the unique linear extension is a shelling order by assumption. Let $h>2$.
    By deciding that this linear extension corresponds to the identity permutation $e\in S_h$, any other linear extension of $Y$ corresponds to a permutation $\sigma \in S_h$. 
    So let $l\mapsto e$, where $e=123\ldots h \in S_h$. Hence the set of linear extensions of $Y$ is in bijection with a set $L\subseteq S_h$. 
    We prove, by induction on $\ell(w)$, that $w$ is a shelling order, for all $w\in L$. 
    If $\ell(w)=0$ then $w=e$ and $e$ is a shelling order by assumption. 
    Let $\ell(w)>0$ and $s_i\in D_R(w)$, i.e. $w(i+1)<w(i)$.
    We claim that $ws_i$ is a linear extension. In fact, the facets $w(i)\in [n]^{d+1}_<$ and $w(i+1)\in [n]^{d+1}_<$ are incomparable in the Bruhat order $\leq_B$ on $[n]^{d+1}_<$. To see this, assume by contradiction that $w(i)$ and $w(i+1)$ are comparable. This implies that $w(i)<_B w(i+1)$, because $w$ is a linear extension of $\leq_B$. Therefore $e(w(i+1))=w(i+1)<w(i)=e(w(i))$ and $e(w(i+1))=w(i+1)>_B w(i)=e(w(i))$. This is in contradiction to $e$ being a linear extension of the Bruhat order. Hence, $w(i)$ and $w(i+1)$ are incomparable in the Bruhat order and, by our inductive hypothesis, $ws_i$ is a shelling order. Then we can apply Proposition \ref{lemma che non sa nessuno} to deduce that $w$ is a shelling order.
 \end{proof}
Theorem \ref{teorema tutti shelling} provides an alternative proof of \cite[Theorem 3.4]{BoloSentiBulletin}. In fact, by \cite[Theorem 4.8]{samper}, the lexicographic order is a shelling order for the complexes with the quasi--exchange property.
\begin{cor}
    Let $X\subseteq [n]^k_<$ with the quasi--exchange property. Then any linear extension of the Bruhat order of $X$ is a shelling order.
\end{cor}

In Figure \ref{Hasse0} is depicted the Hasse diagram of the Bruhat order of a simplicial complex in $[8]^4_<$ which is not linearly shellable, as was checked by using SageMath \cite{sagemath}. In the next section we provide a wide class of linearly shellable complexes arising from Coxeter groups. 

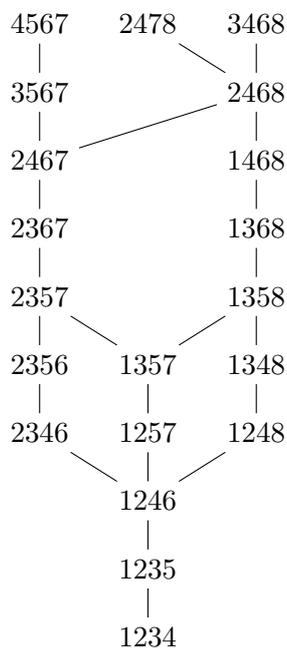
\begin{figure} \begin{center}\begin{tikzpicture}
\matrix (a) [matrix of math nodes, column sep=0.4cm, row sep=0.4cm]{
        & 4567    & 2478  & 3468        & \\
        & 3567    &       & 2468    & \\
        & 2467    &       & 1468    & \\
        & 2367    &       & 1368    & \\
        & 2357    &       & 1358    & \\
    &  2356       & 1357  & 1348        &  \\
   &  2346        & 1257  & 1248         & \\
        &         & 1246  &         & \\
        &         & 1235  &         & \\
        &         & 1234  &         & \\};

\foreach \i/\j in {10-3/9-3, 9-3/8-3, 8-3/7-2, 8-3/7-3,
8-3/7-4, 7-2/6-2, 7-3/6-3, 7-4/6-4, 6-2/5-2, 6-3/5-2,
6-3/5-4, 6-4/5-4, 5-2/4-2, 4-2/3-2, 3-2/2-2, 2-2/1-2,
5-4/4-4, 4-4/3-4, 3-4/2-4, 2-4/1-3, 2-4/1-4,
3-2/2-4}
    \draw (a-\i) -- (a-\j);
\end{tikzpicture} \caption{Hasse diagram of a shellable complex which is not linearly shellable.} \label{Hasse0} \end{center} \end{figure} 

\section{Coxeter complexes and linearity}\label{sezione Coxeter}

In \cite{bjorner} Bj\"{o}rner studied combinatorial properties of the so called Coxeter complex of a Coxeter group $W$. Among other results, he provides  shelling orders of such simplicial complex obtained from linear extensions of the weak order on $W$. In this section we consider Coxeter complexes of finite subsets of a Coxeter group $W$, obtaining shelling orders by any linear extension of a refinement of the Bruhat order on order ideals of $(W,\leq_R)$ and intervals $[u,v]_R$. These include lower Bruhat intervals of any quotient ${^JW}$. In the last part of the section we also consider type--selected Coxeter complexes of finite order ideals of quotients $W^J$, generalizing to Coxeter groups the result of \cite[Theorem 4.5]{BoloSentiBulletin}, and two--sided Coxeter complexes.
We prove that all such complexes are  linearly shellable, then providing a large class of complexes with this property.

We consider a Coxeter system
$(W,S)$; given $J\subseteq S$ finite and $w \in W$, we let $$P_J(w):=\left\{(P^{(s)}(w),s): s\in J \right\} \subseteq W\times J.$$
For $Y\subseteq W$, we define $$S(Y):=\bigcup_{y\in Y}\{s\in S: s\leq y\}.$$ Clearly if $Y$ is finite then $S(Y)$ is finite.
We now introduce the Coxeter complex of any finite subset of $W$. This extends to all Coxeter groups the construction of \cite[Section 4]{BoloSentiBulletin}, suitable for the barycentric subdivision of a pure simplicial complex.
\begin{dfn}
    Let $Y\subseteq W$ be finite. The \emph{Coxeter complex} of $Y$ is the pure simplicial complex of dimension $|S(Y)|-1$ whose set of facets is
    $C(Y):=\{P_{S(Y)}(y): y\in Y\}$.
\end{dfn} When $Y$ is an interval $[u,v]_R$ of the right weak order on $W$ we denote $C(Y)$ by $C(u,v)$. 
Our next aim is to prove that Coxeter complexes of weak intervals and of finite order ideals of the right weak order are linearly shellable. To do this we need the
following lemma as preliminary result, whose proof follows the ideas of \cite[Theorem 2.1]{bjorner}.
\begin{lem} \label{lemma1}
    Let $u,v,w\in W$ be such that $u\leq_R v$
    and $u\leq_R w$. If $P^{(s)}(v)=P^{(s)}(w)$ for all $s\in D_R(u^{-1}v)$, then $v\leq_R w$.
\end{lem}
\begin{proof} Let $v=uz$ and $w=uz'$, with $\ell(v)=\ell(u)+\ell(z)$ and $\ell(w)=\ell(u)+\ell(z')$.
If $P^{(s)}(v)=P^{(s)}(w)$ for all $s\in D_R(u^{-1}v)$, then
$w^{-1}v \in W_{S\setminus \{s\}}$, for all $s\in D_R(z)$.
Hence $w^{-1}v\in W_{S\setminus D_R(z)}$, i.e.
$zy=z'$ for some $y\in W_{S\setminus D_R(z)}$.
Since $z\in W^{S\setminus D_R(z)}$, we have that 
$\ell(zy)=\ell(z)+\ell(y)$, i.e. $z\leq_R z'$; 
this implies $v\leq_R w$.
\end{proof}
The statement of the following proposition is the analogous, for Coxeter complexes of weak intervals, of the fact that as posets $[u,v]_R \simeq [e,u^{-1}v]_R$.
\begin{prop} \label{isomorfismo complessi}
    Let $u \leq_R v$. Then, as simplicial complexes, $C(u,v)\simeq C(e,u^{-1}v)$.
\end{prop}
\begin{proof}
 Let $\phi: C(u,v) \rightarrow C(e,z)$ be the function defined by the assignment $(P^{(s)}(w),s)\mapsto (P^{(s)}(u^{-1}w),s)$, for all $s\in S$, $w\in [u,v]_R$. It holds that 
 $P^{(s)}(x)=P^{(s)}(y)$ if and only if $P^{(s)}(u^{-1}x)=P^{(s)}(u^{-1}y)$, for all $x,y\in [u,v]_R$. In fact, if we let $z_x:=u^{-1}x$ and $z_y:=u^{-1}y$, we have that
 \begin{eqnarray*}
     P^{(s)}(z_x)=P^{(s)}(z_y) &\Leftrightarrow& z_y^{-1}z_x \in W_{S\setminus \{s\}} \\
     &\Leftrightarrow& y^{-1}x\in W_{S\setminus \{s\}} \\
      &\Leftrightarrow& P^{(s)}(x)=P^{(s)}(y),
 \end{eqnarray*} for all $s\in S$, $x,y\in [u,v]_R$. Then the function $\phi$ is an isomorphism of simplicial complexes.
\end{proof}
Now provide a relation between linear extensions of the weak order and shelling orders of Coxeter complexes of order ideals and weak intervals.
\begin{prop} \label{prop shelling debole ideali}
    Let $I\subseteq W$ be an order ideal or an interval of the poset $(W,\leq_R)$. If $L:=(w_1,w_2,\ldots)$ is a linear extension of $I$, then the tuple $(P_{S(I)}(w_1),P_{S(I)}(w_2),\ldots)$
    is a shelling order for $C(I)$.
\end{prop}
\begin{proof} 
    The shellability of $I$ is a consequence of \cite[Theorem 2.1]{bjorner} and Proposition \ref{isomorfismo complessi}. If $I$ is an order ideal, then there exists a linear extension of $(W,\leq_R)$ whose first $p$ elements are the elements of $I$ and then we conclude again by \cite[Theorem 2.1]{bjorner}. When $I$ is an interval the result follows by the previous point and the isomorphism of Proposition \ref{isomorfismo complessi}.
\end{proof}
In the proof of the next theorem we introduce a refinement $\preccurlyeq$ of the Bruhat order on $W$. For any $Y\subseteq W$ and a suitable labeling of the vertices of the simplicial complex $C(Y)$, it turns out that the poset $(Y,\preccurlyeq)$ is isomorphic to the Bruhat order on $C(Y)$ (we are considering simplicial complexes as sets of Grassmannian permutations). In Figures \ref{Hasse1} and \ref{Hasse2} are depicted the Hasse diagrams of $([s_1s_2,w_0]_R,\leq_R)$ and $([s_1s_2,w_0]_R,\preccurlyeq)$, respectively; we are in type $A_3$, with generators $S=\{s_1,s_2,s_3\}$. The element $w_0$ is the one of maximal length. In these figures the elements of $W$ are permutations in one--line notation. 
\begin{figure} \begin{center}\begin{tikzpicture}
\matrix (a) [matrix of math nodes, column sep=0.4cm, row sep=0.6cm]{
           & 4321    & \\
     4231   &     & 3421\\
     2431   &     & 3241\\
     2341      &  & 3214\\
           & 2314 &\\};
\foreach \i/\j in {5-2/4-1,
5-2/4-3, 4-1/3-1, 4-1/3-3, 4-3/3-3, 3-1/2-1, 3-3/2-3, 2-1/1-2, 2-3/1-2}
    \draw (a-\i) -- (a-\j);
\end{tikzpicture} \caption{Hasse diagram of $([u,w_0]_R,\leq_R)$.} \label{Hasse1} \end{center} \end{figure}
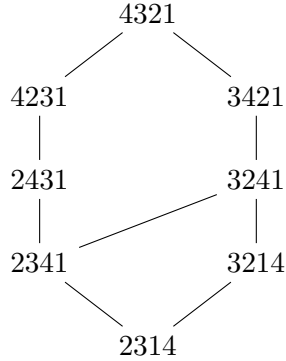
\begin{thm} \label{teorema lineare} Let $I \subseteq W$ be a finite order ideal of $(W,\leq_R)$. Then 
    $C(I)$ is
    linearly shellable.
\end{thm}
\begin{proof} For $Y\subseteq W$ and $s\in S$ we define $P^{(s)}(Y):=\{P^{(s)}(y):y\in Y\}$.
 Let $k:=|S(I)|$, $s_1<\ldots<s_k$ a total order on $S(I)$ and define $P(w):=P_{S(I)}(w)$ for shortness. Let $L_s$ be a linear extension of the Bruhat order on $P^{(s)}(I)$, for all $s\in S$, and
    $h:[k]\rightarrow \N$ be the function defined by setting $h(i)=|P^{(s_i)}(I)|$, for all $i\in [k]$.
    Then we have a total order
    $$(w^1_1,s_1),\ldots ,(w^1_{h(1)},s_1), (w^2_1,s_2) , (w^2_2,s_2), \ldots , (w^k_{h(k)},s_k)$$
    such that $L_{s_i}=(w^i_{1},\ldots,w^i_{h(i)})$, for all $i\in [k]$. This induces a bijection $\sigma : \bigcup_{w\in I}P(w) \rightarrow [n]$, where $n:=|\bigcup_{w\in I}P(w)|$. We then define a function $\sigma_I: I \rightarrow [n]^k_<$ by setting
    $$\sigma_I(w)=\left(\sigma(P^{(s_1)}(w)),\ldots, \sigma(P^{(s_k)}(w))\right)\in [n]^k_<,$$ for all $w\in I$,
    and a partial order on $I$ by setting $x\preccurlyeq y$ if and only if $\sigma_I(x)\leq \sigma_I(y)$, for all $x,y\in I$ (the antisymmetry follows by Corollary \ref{deodhar} for $J=\varnothing$). Since the Bruhat order is a refinement of the weak order, and by Corollary \ref{deodhar} for $J=\varnothing$, we have injective morphisms of posets $$(I,\leq_R)\hookrightarrow (I,\leq)\hookrightarrow (I,\preccurlyeq).$$ This implies that any linear extension of $(I,\preccurlyeq)$
    is a linear extension of $(I,\leq)$ and $(I,\leq_R)$.
Let $p:=|I|$ and $C=(w_1,\ldots,w_p)$ be a linear extension of $(I,\preccurlyeq)$. By Proposition \ref{prop shelling debole ideali}, $(P(w_1),\ldots, P(w_p))$ is a shelling order for $I$.
We have that $\sigma_I:(I,\preccurlyeq) \rightarrow \imm(\sigma_I)$ is an isomorphism of posets; hence $(\sigma_I(w_1),\ldots,\sigma_I(w_p))$ is a linear extension of $\imm(\sigma_I)\subseteq [n]^k_<$ which is a shelling order, i.e. $I$ is linearly shellable.
\end{proof}

\begin{figure} \begin{center}\begin{tikzpicture}
\matrix (a) [matrix of math nodes, column sep=0.4cm, row sep=0.6cm]{
           & 4321    & \\
     4231   &     & 3421\\
     2431   &     & 3241\\
     2341      &  & 3214\\
           & 2314 &\\};
\foreach \i/\j in {5-2/4-1,
5-2/4-3, 4-1/3-1, 4-1/3-3, 4-3/3-3, 3-1/2-1, 3-1/2-3,  3-3/2-3, 3-3/2-1, 2-1/1-2, 2-3/1-2}
    \draw (a-\i) -- (a-\j);
\end{tikzpicture} \caption{Hasse diagram of $([u,w_0]_R,\preccurlyeq)$.} \label{Hasse2} \end{center} \end{figure}
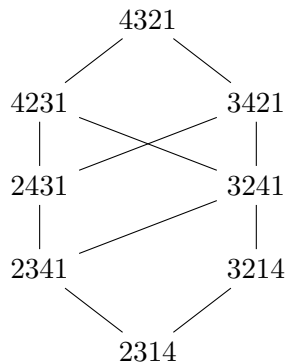
\begin{oss} \label{osservazione Bruhat}
    In the proof of Theorem \ref{teorema lineare} we have used the poset injection 
$([u,v]_R,\leq) \hookrightarrow ([u,v]_R,\preccurlyeq)$. The opposite injection of posets is not true in general. For example, in $S_4$ the posets $([e,4231]_R,\leq)$ and $([e,4231]_R,\preccurlyeq)$ are not isomorphic. In the case of Figure \ref{Hasse2} we have that $([u,v]_R,\leq)\simeq ([u,v]_R,\preccurlyeq)$.
\end{oss}
Our first corollary of the previous theorem concerns weak intervals. We are using the fact that linear shellability is preserved under isomorphism of simplicial complexes.
\begin{cor} \label{corollario intervalli}
Let $u,v\in W$ be such that $u\leq_R v$. Then the Coxeter complex of $[u,v]_R$ is linearly shellable. 
\end{cor}
\begin{proof}
    Since any interval $[e,v]_R$ is an order ideal of $(W,\leq_R)$, the result follows by Theorem \ref{teorema lineare} and Proposition \ref{isomorfismo complessi}.
\end{proof}
Since any order ideal of the right weak order of a quotient ${^JW}$ is an order ideal of $(W,\leq_R)$, by Theorem \ref{teorema lineare} we obtain a second corollary.
\begin{cor} \label{corollario quozienti} Let $J\subseteq S$. Then 
    the Coxeter complex of a finite order ideal of $({^JW},\leq_R)$ is linearly shellable. In particular, this is true for lower Bruhat intervals of ${^JW}$.  
\end{cor}
\begin{oss}
    We observe that the Coxeter complex of a Bruhat interval could be not shellable. For example,  let $u=2134$ and $v=4231$ in $S_4$. The simplicial complex $C([u,v])$ has $14$ facets and dimension $2$. Its $h$-vector is $(1,9,5,-1)$ and then it is not shellable.
\end{oss}
For the definition of \emph{Coxeter cones} we refer to \cite{stembridge}. They are Coxeter complexes of order ideals of the right weak order of a finite Coxeter group $W$ (see \cite[Proposition 2.6]{stembridge}); hence we obtain a third corollary.
\begin{cor} \label{cor coni}
A Coxeter cone is linearly shellable. 
\end{cor}
By Corollary \ref{cor deodhar1} for $J=\varnothing$, the simplicial complex $C(I)$ is either \emph{thin} or \emph{subthin}, i.e. either every face of dimension $|S(I)|-2$ is contained in exactly two facets or it is not thin and every face of dimension $|S(I)|-2$ is contained in at most two facets. From this fact and the shellability of $C(u,v)$ we obtain the next corollary (see e.g. \cite[Appendix A2]{BB}). 
\begin{cor} Let $I \subseteq W$ be a finite order ideal or an interval of $(W,\leq_R)$. Then the geometric realization of the Coxeter complex of $I$ is $PL$-homeomorphic to 
\begin{enumerate}
    \item the sphere ${\SSS}^{|S(I)|-1}$, if $I=W$;
    \item the ball ${\BBB}^{|S(I)|-1}$, otherwise.
\end{enumerate}
\end{cor}
We now extend to all Coxeter groups the result of \cite[Theorem 4.5]{BoloSentiBulletin}. The cited theorem is related to barycentric subdivisions, in particular of lower Bruhat intervals in maximal quotients of $S_n$, i.e. of independence complexes of Schubert matroids (see \cite{fan} and references therein for this class of matroids). Let $J\subseteq S$ and $Y\subseteq W$ be a finite set; we define
$$C_J(Y):=\{P_{J\cap S(Y)}(w): w\in Y\} \subseteq \mathcal{P}(W \times J).$$
Following \cite{bjorner}, we call the simplicial complex $C_J(Y)$ a \emph{type--selected Coxeter complex} of $Y$. This is a pure simplicial complex of dimension $|J\cap S(Y)|-1$.
\begin{thm}\label{teorema selected}
Let $J\subseteq S$ and $I\subseteq W^J$ be a finite order ideal of $(W^J,\leq)$. Then any linear extension of $I$ is a shelling order for
$C_{S\setminus J}(I)$. Moreover $C_{S\setminus J}(I)$ is linearly shellable.    
\end{thm}
\begin{proof}
  Define $K:=(S\setminus J)\cap S(I)$ and  $k:=|I|$. Let $L:=(w_1,\ldots,w_k)$ be a linear extension of $I$. For $1< j \leq k$, let $w:=w_j$. Let $i\in [k-1]$ and assume $|P_K(w_i)\cap P_K(w)|<|K|-1$. By Corollary \ref{cor deodhar1} we have that $r\in D_R(w)$ implies $P^{J\cup \{r\}}(w)\neq P^{J\cup \{r\}}(w_i)$. So consider $r\in D_R(w)$ and define $u:=P^{J\cup \{r\}}(w)$. Then $u<w$ and $u\in W^J$. Since $I$ is an order ideal of the Bruhat order, we have that $u\in I$. Therefore there exists $h\in [k]$ such that $u=L_h$.
By Corollary \ref{cor deodhar1} we have that $P_K(u)= P_K(w) + \{(P^{(r)}(w),r),(P^{(r)}(u),r)\}$. Then $(P_K(w_1),\ldots,P_K(w_k))$ is a shelling order.
The linear shellability of $C_{S\setminus J}(I)$ follows as in the proof of Theorem \ref{teorema lineare}.
\end{proof}
We state now a formula for the $h$-polynomials of the complexes of this section. By Proposition \ref{isomorfismo complessi}, it includes also the $h$-polynomials of weak intervals.
\begin{cor}\label{h-polynomial} Let $I\subseteq W$ be a finite set. Then the polynomial 
$$h(I)=\sum\limits_{w\in I}q^{|D_R(w)|}$$
is the $h$-polynomial of:
\begin{enumerate}
    \item the Coxeter complex $C(I)$, if $I\subseteq W$ is a finite order ideal of $(W,\leq_R)$;
    \item the type--selected Coxeter complex $C_{S\setminus J}(I)$, if $J\subseteq S$ and $I\subseteq W^J$ is a finite order ideal of $(W^J,\leq)$.
\end{enumerate}
\end{cor}
\begin{proof}
    We prove the result in the type--selected case. We refer to \cite[Appendix A2]{BB} for definitions and results about $h$-polynomials of shellable complexes and restrictions of facets. Define $K:=(S\setminus J)\cap S(I)$ and  $k:=|I|$.
Let $L=(w_1,\ldots,w_k)$ be a linear extension of the Bruhat order of $I$. By Theorem \ref{teorema selected}, $L$ is a shelling order. Let $j\in [k]$ and define  $w:=w_j$. We claim that the a vertex $(P^{(r)}(w),r)$ lies in the restriction of the facet $P_K(w)$ if and only if $r\in D_R(w)$. From the proof of Theorem \ref{teorema selected} we deduce that if $r \in D_R(w)$ then $(P^{(r)}(w),r)$ is a vertex of the restriction of $P_K(w)$. So assume  $r\not \in D_R(w)$; this implies $P^{J\cup \{r\}}(w)=w$. Let $u\in I\setminus \{w\}$ such that $P_K(w)\cap P_K(u)=P_K(w)\setminus \{(P^{(r)}(w),r)\}$. By Corollary \ref{cor deodhar1} this is equivalent to $P^{J\cup \{r\}}(u)=w$. Then $u=u^{J\cup \{r\}}u_{J\cup \{r\}}=wu_{J\cup \{r\}}>w$, i.e. $u=L_k$, for some $k>j$. Hence $(P^{(r)}(w),r)$ does not lie in the restriction of $P_K(w)$. Therefore the cardinality of the restriction of the facet $P_K(w)$ coincides with $|D_R(w)|$ and the result follows.
\end{proof}

\begin{ex} \label{esempio selezionato}
    Let $S=\{s_1,s_2,s_3\}$, $J:=\{s_2\}$, $K:=S\setminus J$, $4132 \in S_4^J$ and $I:=\{z\in S_4^J: z\leq 4132\}$. A linear extension of $I$ is $$1234,2134,1243,2143,1342,3124,4123,3142,4132.$$ The complex $C_K(I)$ has $7$ vertices and $9$ facets. Its $h$-polynomial is $1+5q+3q^2$. Notice that $C_K(I)$ is neither thin nor subthin. In fact  
    \begin{itemize}
        \item $P_K(1234)=\{(1234,s_1),(1234,s_3)\}$,
        \item $P_K(2134)=\{(2134,s_1),(1234,s_3)\}$,
        \item $P_K(3124)=\{(3124,s_1),(1234,s_3)\}$.
    \end{itemize} The simplicial complex $C_K(I)$, which is a graph, is depicted in Figure \ref{grafo}.
\end{ex}
\begin{figure} \begin{center}\begin{tikzpicture}
\matrix (a) [matrix of math nodes, column sep=0.9cm, row sep=0.4cm]{
           & \bullet    & \\
     \bullet   &     & \bullet\\
        &  \bullet   & \\
     \bullet      &  & \bullet\\
           & \bullet &\\};
\foreach \i/\j in {1-2/2-1, 1-2/2-3, 1-2/3-2, 2-1/4-1,2-3/4-3,3-2/4-1,3-2/4-3, 4-1/5-2, 4-3/5-2}
    \draw (a-\i) -- (a-\j);
\end{tikzpicture} \caption{The simplicial complex $C_K(I)$ of Example \ref{esempio selezionato}.} \label{grafo} \end{center} \end{figure}
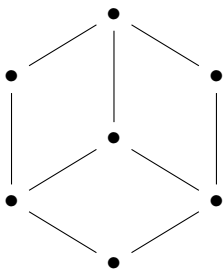

The last theorem of this section concerns linear shellability of \emph{two--sided
Coxeter complexes}. We refer to \cite{petersen} for definitions. We only observe that Bruhat orders are refinements of \emph{two--sided weak orders} (see also \cite[Exercise 3.8]{BB}) and that linear extensions of two--sided weak orders provide shelling order for two--sided Coxeter complexes, by \cite[Theorem 3]{petersen}. Having in mind these facts, the proof of the theorem is analogous to the one of Theorem \ref{teorema lineare} and it is omitted.
\begin{thm} \label{teorema two-side}
   The two--sided Coxeter complex of a finite Coxeter system is linearly shellable. 
\end{thm}

\section{Linear strong shellability}\label{sezione strong}
In \cite{coreani} has been introduced the notion of \emph{strong shellability}. We recall the definition (see also \cite[Definition 2.2]{coreani}). Let $h\in \N\setminus \{0\}$; an element $C\in \mathrm{Conf}_h([n]^k_<)$ is a \emph{strong shelling order} if $i,j\in [h]$, with $i<j$, implies that
there exists $r<j$ such that
$C_r=C_j+\{a,b\}$, with $a\in C_j\setminus C_i$ and $b\in (C_r \cap C_i)\setminus C_j$.
A pure simplicial complex is strongly shellable if there exists a strong shelling order on its facets. Now we introduce the notion of linear strong shellability. 
\begin{dfn}
    Let $X$ be a pure simplicial complex of dimension $d\in \N$ with $n\in \N\setminus \{0\}$ vertices.  
    We say that $X$ is \emph{linearly strongly shellable} if there exists 
    a bijective function $\sigma : V(X)\rightarrow [n]$ which induces 
    a function $\sigma_X : X\rightarrow [n]^{d+1}_<$ such that the induced subposet $\imm(\sigma_X) \subseteq [n]^{d+1}_<$
    has a linear extension which is a strong shelling order.
\end{dfn}
The following statement is the analogue of \cite[Proposition 5.3]{BoloSentiBulletin} for strong shelling orders.
\begin{prop}\label{lemma che non sa nessuno strong}
Let $C \in \mathrm{Conf}_h([n]^k_<)$ be a strong shelling order, with $h \geqslant 3$. If $|C_{h-1} \cap C_h|<k-1$ then $(C_1,\ldots,C_h,C_{h-1})$ is a strong shelling order.
\end{prop}
\begin{proof}
Consider $i<h-1$. For the pair $(C_i,C_{h-1})$ we have nothing to show. For the pair $(C_i,C_h)$, there exists $x \in C_h \setminus C_i$ and $j<h$ such that $C_j=C_h +\{x,y\}$, for some $y\in C_i$. By our assumption, $j \neq h-1$ and the strong shellability condition on this pair follows.

It remains to verify the strong shellability condition for $(C_h,C_{h-1})$. By the fact that $C$ is a strong shelling order and our assumption, there exists $z \in C_h \setminus C_{h-1}$ and $j<h-1$ such that $C_j=C_h+\{z,y\}$, for some $y\in C_{h-1}$. Since $C$ is a strong shelling order, there exists $c \in C_{h-1} \setminus C_j$ and $r<h-1$ such that $C_r=C_{h-1}+\{c,v\}$, for some $v\in C_j$. We have that $c \notin C_j=C_h+\{z,y\}$ and $c \neq z$; then $c \in C_{h-1} \setminus C_h$ and $C_r=C_{h-1}+\{c,v\}$, with $r<h-1$. Moreover $v\in C_j=C_h+\{z,y\}$. In fact, if $v\neq y$ then $v\in C_h$. If $v=y$ then $v\in C_{h-1}$, a contradiction. Then $v\in C_h$ and this completes the proof.
\end{proof}
Proposition \ref{lemma che non sa nessuno strong} has some interesting consequences, as for standard shellability. First of all we obtain the analogue of Theorem \ref{teorema tutti shelling} for strongly shellable complexes, whose proof we omit.
\begin{thm} \label{teorema strong} Let $X$ be a pure simplicial complex of dimension $d\in \N$ with $n\in \N\setminus \{0\}$ vertices. 
    The following are equivalent:
    \begin{enumerate}
        \item[1)] $X$ is linearly strongly shellable;
        \item[2)] there exists a bijective function $\sigma : V(X)\rightarrow [n]$ which induces a function $\sigma_X: X\rightarrow [n]^{d+1}_<$ such that $<_{lex}$ on $\imm(\sigma_X)$ is a strong shelling order;
        \item[3)] there exists a bijective function $\sigma : V(X)\rightarrow [n]$ which induces a function $\sigma_X: X\rightarrow [n]^{d+1}_<$ such that every linear extension of $\imm(\sigma_X)$ is a strong shelling order. 
    \end{enumerate}
\end{thm}
The lexicographic order is a strong shelling order for \emph{pure shifted complexes}, by \cite[Proposition 6.5]{coreani}. Since the pure shifted complexes are exactly the order ideals of $[n]^k_<$, we obtain the next corollary.
\begin{cor}
    The order ideals of $[n]^k_<$ are linearly strongly shellable.
\end{cor}
The \emph{reverse lexicographic order} on $[n]^k_<$ is defined by letting $x<_{revlex} y$ if and only if $x_m<y_m$, where $m:=\max\{i\in [k]: x_i\neq y_i\}$. The reverse lexicographic order is a linear extension of the Bruhat order on $[n]^k_<$ and it is a strong shelling order for independence complexes of matroids, by \cite[Proposition 6.3]{coreani}. This implies the next result.
\begin{cor}
The independence complex of a matroid is linearly strongly shellable.
\end{cor}
As in \cite{BoloSentiBulletin}, Proposition \ref{lemma che non sa nessuno strong} enable us to define \emph{promotion} and \emph{evacuation} operators for strong shelling orders.
We refer to \cite[Section 5]{BoloSentiBulletin} for notation involved in the following corollary.
\begin{cor}
 Let $h\in \N\setminus \{0\}$ and $C\in \mathrm{Conf}_h([n]^k_<)$ be a strong shelling order. Then the promotion $\partial_D C$ and the evacuation $\epsilon_D C$ are strong shelling orders.
\end{cor}

\section{Declarations}
\begin{itemize}
    \item Ethical Approval: Not applicable.
    \item Funding: None.
    \item Availability of data and materials: Not applicable.
\end{itemize}


\begin{thebibliography}{9}

\bibitem{ardila}
F. Ardila, F. Castillo and J. A. Samper, {\em The topology of the external activity complex of a matroid}, Electronic Journal of Combinatorics  23, 3 (2016).


\bibitem{bjorner}
A. Bj\"{o}rner, {\em Some combinatorial and algebraic properties of Coxeter complexes and Tits buildings},  Advances in Mathematics 52.3, 173-212 (1984).

\bibitem{BB}
A. Bj\"{o}rner and F. Brenti, {\em Combinatorics of Coxeter Groups},
Graduate Texts in Mathematics, 231, Springer-Verlag, New York, 2005.


\bibitem{BoloSentiBulletin}
D. Bolognini and P. Sentinelli, {\em Linear extensions and shelling orders},  Bulletin of the London Mathematical Society (2023).

\bibitem{CarnevaleSentinelliDyer}
A. Carnevale, M. Dyer, and P. Sentinelli, {\em The intermediate orders of a Coxeter group},
Proceedings of the American Mathematical Society 151.04 1433-1443, (2023).

\bibitem{fan}
N. JY Fan and Yao Li, {\em On the Ehrhart polynomial of Schubert matroids},
Discrete \& Computational Geometry 71.2, 587-626 (2024).

\bibitem{coreani}
J. Guo, Y.-H. Shen, and T. Wu, {\em Strong shellability of simplicial complexes}, Journal of the Korean Mathematical Society 56.6, 1613-1639  (2019).

\bibitem{Hachimori}
M. Hachimori, {\em Simplicial complex library,\\} \href{https://infoshako.sk.tsukuba.ac.jp/~hachi/math/library/nonextend_eng.html}{https://infoshako.sk.tsukuba.ac.jp/~hachi/math/library/nonextend}

\bibitem{lacina}
S. Lacina, {\em Maximal chain descent orders}, arXiv preprint arXiv:2209.15142 (2022).

\bibitem{petersen}
T. K. Petersen, {\em A two--sided analogue of the Coxeter complex}, The Electronic Journal of Combinatorics, 25, 4 (2018).

\bibitem{sagemath}
{The Sage Developers}, \emph{{S}age{M}ath, the {S}age {M}athematics {S}oftware {S}ystem ({V}ersion 7.3)}, {\tt https://www.sagemath.org}.

\bibitem{samper}
J. A. Samper, {\em Quasi--matroidal classes of ordered simplicial complexes}, Journal of Combinatorial Theory, Series A 175, 105274 (2020).


\bibitem{Stanley}
R. P. Stanley, {\em Enumerative Combinatorics}, Vol. 1, Wadsworth and
Brooks/Cole, Monterey, CA, 1986.

\bibitem{stembridge}
J. R. Stembridge, {\em Coxeter cones and their h-vectors}, Advances in Mathematics 217.5, 1935-1961 (2008).

\bibitem{wachs}
M. L. Wachs, {\em Quotients of Coxeter complexes and buildings with linear diagram}, European Journal of Combinatorics 7.1, 75-92 (1986).


\end{thebibliography}
\end{document}